\RequirePackage{fix-cm}
\documentclass[smallextended]{svjour3}       
\smartqed  
\usepackage{graphicx}
\usepackage{amsmath}
\usepackage{amssymb}
\usepackage{enumerate}

\allowdisplaybreaks
\begin{document}

\title{H\"older continuous  solutions to the complex Monge-Amp\`ere equations in non-smooth  pseudoconvex  domains
\thanks{This work is finished during the first author's post-doctoral fellowship of the Vietnam Institute for Advanced Study in Mathematics. He wishes to thank the institution for their kind hospitality and support. 
This research is funded by Vietnam National Foundation for Science and Technology Development (NAFOSTED) under grant number 101.02-2016.06. 
The authors would like to thank the referees for valuable remarks which lead to the improvements of the exposition of the paper.}
}

\titlerunning{Solutions to complex Monge-Amp\`ere   equations}        

\author{Nguyen Xuan Hong           \and
        Tran Van Thuy 
}


\institute{N. X. Hong \at
              Department of Mathematics\\ 
              Hanoi National University of Education\\ 
              136 Xuan Thuy Street, Caugiay District, Hanoi, Vietnam \\ 
              \email{xuanhongdhsp@yahoo.com}           
           \and
           T. V. Thuy \at
           \email{thuyhum@gmail.com}  
}

\date{Received: date / Accepted: date}

\maketitle

\begin{abstract}
In this paper, we prove  the H\"older continuity for  solutions to the complex  Monge-Amp\`ere equations on non-smooth pseudoconvex   domains of plurisubharmonic  type ${m}$. 
\end{abstract}

\section{Introduction} 


Let $\Omega$ be an open set in $\mathbb C^n$. An upper semi-continuous function $u:\Omega\rightarrow[-\infty, +\infty)$ is called plurisubharmonic on $\Omega$ if for every complex line $l$ of $\mathbb C^n$,  $u|_{l \cap\Omega}$ is a subharmonic function in $l\cap \Omega$. The set of plurisubharmonic functions on $\Omega$ is denoted by $PSH(\Omega)$.
According to the fundamental work of Bedford and Taylor  \cite{BT82} (also see \cite{BT76}),  the complex Monge-Amp\`ere operator $(dd^ c. )^n$ is well-defined over the class of locally bounded plurisubharmonic functions.  Cegrell \cite{Ce04} introduced a classes of unbounded plurisubharmonic functions on bounded hyperconvex domain on which the complex Monge-Amp\`ere can be defined.  

The Dirichlet problem for the complex Monge-Amp\`ere equation is one of the important and central problems of pluripotential theory. 
Assume that  $\phi \in \mathcal C(\partial{\Omega})$,  $0\leq f\in L^p(\Omega)$ with $p>1$.  The Dirichlet problem is the problem of finding a  function $u$ satisfying: 
$$
MA (\Omega, \phi, f): 
\begin{cases}
u\in PSH(\Omega)\cap L^\infty (   \Omega),
\\ (dd^c u)^n =f dV ,
\\ \lim_{z\to \xi}u(z) =\phi(\xi),\  \forall  \xi \in \partial \Omega.
\end{cases}
$$ 

When $\Omega$ is a smooth, bounded strictly pseudoconvex domain in $\mathbb C^n$, there are some known results for the existence and regularity for this problem due to   \cite{BT76}, \cite{BT82} and \cite{CKNS85}.  Bedford and  Taylor \cite{BT76} proved that if    $\phi \in \mathcal  C^{2\alpha}(\partial \Omega)$  with $0 < \alpha \leq 1$ and $f^{\frac{1}{n}} \in \mathcal C^\alpha (\overline \Omega)$  then $MA(\Omega, \phi, f)$ has a unique plurisubharmonic solution $u \in \mathcal C^\alpha (\overline \Omega)$.
Next, Bedford and Tayloy \cite{BT82}  showed that if $f\in \mathcal C(\overline \Omega)$ then  there exists a continuous solution $u$ on $\overline{\Omega}$. Later, Caffarelli, Kohn, Nirenberg and Spruck \cite{CKNS85} studied the  global regularity. They proved that if $f \in \mathcal C^\infty(\overline{\Omega})$  is strictly positive and $\phi \in \mathcal C^\infty(\partial{\Omega})$   then $MA(\Omega, \phi,f)$ has a unique plurisubharmonic solution $u  \in \mathcal C^\infty (\overline{\Omega})$. 
 
When  $\Omega$ is a non-smooth pseudoconvex domain, the problem becomes much more complicated.  B\l ocki  \cite{Bl96} gave a charaterization for the existence of a continuous and plurisubharmonic solution on hyperconvex domains in $\mathbb C^n$. Ko{\l}odziej \cite{Ko96} proved that there exists a unique  continuous solution to $MA (\Omega, \phi, f)$ on strictly pseudoconvex domains.     Li \cite{Li04b} studied the problem on a bounded pseudoconvex domain in $\mathbb C^n$ with $C^2$ boundary. He proved that if $\Omega$ is bounded pseudoconvex domain of plurisubharmonic type $m$  with $C^2$ boundary,  $\phi \in \mathcal  C^{m\alpha}(\partial \Omega)$  with $0 < \alpha \leq \frac{2}{m}$ and $f^{\frac{1}{n}} \in \mathcal C^\alpha (\overline \Omega)$  then $MA(\Omega, \phi, f)$ has a unique solution $u \in \mathcal C^\alpha (\overline \Omega)$. 
Guedj, Ko\l odziej and Zeriahi \cite{GKZ08} studied the problem in   bounded strongly pseudoconvex domains. They showed that if $\phi \in \mathcal C^{1,1} (\partial \Omega)$ then the unique solution $u$ to $MA(\Omega, \phi, f)$ is $\alpha$-H\"older continuous on $\overline \Omega$, for any 
$$0< \alpha \leq   \frac{2}{1+\frac{np}{p-1}}.$$  Cuong \cite{Cuong} generalized the theorem of  \cite{GKZ08} to    complex Hessian equation.
Charabati \cite{Ch15}  proved the H\"{o}lder regularity for solutions to $MA (\Omega, \phi, f)$ in bounded strongly hyperconvex Lipschitz domain. 
Recently, Baracco, Khanh, Pinton and Zampieri \cite{BKPZ} generalized the theorem of  \cite{GKZ08} to $C^2$ smooth bounded pseudoconvex domain of plurisubharmonic type $m$ under the assumption that the boundary data $\phi \in \mathcal  C^{\alpha}(\partial \Omega)$  with $0 < \alpha \leq 2$.  Note that   the technique of \cite{BKPZ} is not  valid when $\Omega$ is not $C^2$ smooth.  

Main purpose of this paper is to generalize the theorem of \cite{BKPZ} from $C^2$ smooth bounded pseudoconvex domain of plurisubharmonic type $m$ to non-smooth pseudoconvex domains of plurisubharmonic type $m$. 
First we give the following definition which is an extension of  Li \cite{Li04b} (also see \cite{BKPZ}).

\begin{definition}\label{de1}
{ \rm 
Let $m>0$ and let $\Omega$ be a pseudoconvex   domain in $ \mathbb C^n$. We say that $\Omega$ is {\em of plurisubharmonic  type ${m}$} if  there exists a bounded  negative  function  $\rho \in \mathcal C^{\frac{2}{m}} (\overline{\Omega})$   such that  $\{\rho<-\varepsilon\}\Subset \Omega$, $\forall  \varepsilon>0$ and  $\rho(z) -|z|^2$ is plurisubharmonic in $\Omega$. 
}\end{definition}

Note that every smooth bounded strictly pseudoconvex domain in $\mathbb C^n$ is of plurisubharmonic type $1$.     
Our main result is the following theorem.

\begin{theorem} \label{the1}
Let $m>0$   and let $\Omega$ be a  pseudoconvex   domain of plurisubharmonic type $m$.  
Let   $\phi \in \mathcal C^{\alpha} (\partial \Omega)$ 
with $0< \alpha \leq 2$ and let  $0\leq f \in L^p(\Omega ) $ with   $p>1$.
Assume that either $\Omega$ is bounded or the support of $f$ is compact on $\Omega$.
Then, there exists a bounded, $\gamma$-H\"older continuous solution  $u(\Omega,\phi,f) $   to $MA(\Omega, \phi,f)$    for all  
$$
0<\gamma <  \min \left(\frac{ \alpha}{2m}, \frac{\alpha}{2} , \frac{1}{2m(1 + \frac{np}{p-1})}  , \frac{1}{2(1 + \frac{np}{p-1})} \right)  .
$$ 
\end{theorem}

The paper is organized as follows. In section 2  prove  that there exists a bounded solution to $MA(\Omega,\phi,f)$.  Section 3 is devoted to prove  Theorem \ref{the1}.

\section{The existence}

Some elements of pluripotential theory that will be used throughout the paper can be found in \cite{ACCH}-\cite{ST08}. 
A bounded  domain $\Omega\subset\mathbb C^n$ is called  hyperconvex  if there exists a bounded plurisubharmonic function $\rho$ such that $\{z\in\Omega: \rho(z)<c\}\Subset\Omega$, for every $c\in (-\infty, 0)$.

First, we have  the following.

\begin{proposition}  \label{Prooooooooooo1234556}
Let $S$ be a subset of $\mathbb C^n$ and let  $\varphi :  S \to \mathbb R$. Assume that $\alpha>0$.  Then, the following statements are equivalent.

(a)   $\varphi$ is $\alpha$-H\"older continuous on $S$, i.e. 
$$
\sup_{ \xi, \zeta \in S, \ \xi \neq \zeta}  \frac{|\varphi (\xi)- \varphi (\zeta)|}{ |\xi -\zeta|^\alpha}<+\infty. 
$$  

(b)   There exist $N, \delta_0 >0 $ such that  $|\varphi(\xi)| \leq N$, $\forall \xi \in S$ and 
$$|\varphi (\xi)- \varphi (\zeta)| \leq N \delta^\alpha, \  \forall \delta \in (0, \delta_0), \  \forall \xi, \zeta \in S, \  |\xi -\zeta|\leq \delta.$$  

The set of all $\alpha$-H\"older continuous functions on $S$ is denoted by $\mathcal C^\alpha(S)$.
\end{proposition} 

\begin{proof}
(a) $\Rightarrow $ (b) is obvious.
We prove (b) $\Rightarrow $ (a). Put
$$
M:= N + 2 \delta_0 ^{-\alpha}  \sup_{z \in S} |\varphi (z)|
.$$
Let $\xi,\zeta \in S$.  If  $|\xi -\zeta |<\delta_0$ then  $$
|\varphi (\xi)- \varphi (\zeta)| \leq N |\xi -\zeta |^\alpha \leq M |\xi -\zeta |^\alpha.
$$
We now assume that  $|\xi -\zeta | \geq  \delta_0$. We have 
\begin{align*}
|\varphi (\xi)- \varphi (\zeta)| \leq  2    \sup_{z \in S} |\varphi (z)| 
\leq M \delta_0^ {\alpha} \leq M |\xi -\zeta |^\alpha.
\end{align*}
Therefore, $|\varphi (\xi)- \varphi (\zeta)| \leq M |\xi -\zeta|^\alpha$  for all $\xi, \zeta \in S$. 
The proof is complete.
\end{proof}

We need the following.

\begin{lemma} \label{lem1}
Let $m>0$ and let  $\Omega$ be a  pseudoconvex   domain of plurisubharmonic type $m$. Let   $\rho$ be as in Definition \ref{de1} and let    $\phi \in \mathcal C^{\alpha} (\partial \Omega)$ 
with $0< \alpha \leq 2$. Define 
$$
M:= \sup_{\xi \in \partial \Omega } |\phi (\xi) | + \sup_{\xi, \zeta \in \partial \Omega, \xi \neq \zeta } \frac{|\phi(\xi) - \phi(\zeta)| }{|\xi-\zeta|^\alpha}$$
and 
\begin{align*}
u =u(\Omega,\phi,0)
:= \sup \{ 
 \varphi \in PSH(\Omega):  
   \varphi \leq\min( \phi(\xi) -h_\xi  , M), 
   \forall \xi \in \partial\Omega\},
\end{align*}
where
$$
h_\xi (z) := - 4 M \left [-  \rho(z) + |z-\xi|^2 \right] ^{\frac{\alpha}{2}}, \ z\in \overline \Omega, \xi \in \partial \Omega.
$$
Then,  
$u$ is  a bounded solution to $MA(\Omega, \phi, 0)$. Moreover, $u \in \mathcal C^{\min(\frac{\alpha}{m}, \alpha)} (\overline{\Omega})$.
\end{lemma} 

\begin{proof}  We use the technique of Li \cite{Li04b}  (also see \cite{BKPZ}).
By the hypotheses it implies that $h_\xi \in PSH(\Omega)$, $\forall \xi \in \partial\Omega$.  
Fix $\zeta, \xi  \in \partial\Omega$ and  $z\in  \overline {\Omega}$. Since $\rho\leq 0$ in $\overline{\Omega}$, $\phi \in \mathcal C^{\alpha}(\partial\Omega)$ and $0<\alpha \leq 2$,  by the definitions of $h_\zeta$ and $h_\xi$, 
\begin{align*}
\phi(\zeta ) + h_\zeta (z)  
& \leq \phi(\xi) + M  |\zeta-\xi|^{\alpha} +h_\zeta(z)
\\ & \leq \phi(\xi) +M [ |z- \zeta|+ |z-\xi| ]^{ \alpha} - 4 M  |z- \zeta| ^{ \alpha}
\\ & \leq \phi(\xi) + 4M   |z- \xi| ^{ \alpha} 
\\&  \leq  \phi(\xi) -h_\xi (z).
\end{align*}
Hence, 
$$
\phi(\zeta) +   h_\zeta 
\leq \phi(\xi) -h_\xi  
\text{ in } \Omega, \ \forall  \zeta, \xi \in \partial \Omega.
$$
Therefore, $-M \leq u \leq M$ 
and 
\begin{equation} \label{2.9.30.6}  
\sup_{\xi \in\partial \Omega} \left[  \phi(\xi) +   h_\xi    \right]
\leq u \leq \inf_{\xi \in\partial \Omega} \left[    \phi(\xi) -  h_\xi   \right]
\text{ on } \overline{ \Omega}.
\end{equation} 
This implies that 
$$
\lim_{z\to \xi}u(z) =\phi(\xi),\  \forall  \xi \in \partial \Omega.
$$
We claim that $u$ is maximal plurisubharmonic in $\Omega$. Indeed, let $G\Subset \Omega$ be an open set and let $v\in PSH(\Omega)$ with $v\leq u$ on $\Omega \backslash G$. Let  $\xi\in \partial \Omega$. Since $h_\xi  \in PSH(\Omega)$ so 
$k_\xi(z):=- \min( \phi(\xi) -h_\xi(z), M) \in PSH(\Omega)$. By   \eqref{2.9.30.6} and using the maximum principle,  
\begin{align*}
\sup_\Omega \left[ v +k_\xi \right] 
= \sup_{\Omega \backslash G} \left[ v +k_\xi \right] 
\leq \sup_{\Omega \backslash G} \left[ u +k_\xi \right] \leq 0.
\end{align*}
Hence, 
$$
v\leq -k_\xi= \min( \phi(\xi) -h_\xi  ,M) \text{ in } \Omega,
$$
for every $\xi \in\partial\Omega$.
By  the definition of $u$ this implies that $v\leq u$ in $\Omega$. Thus, $u$ is a maximal plurisubharmonic function in $\Omega$. This proves the claim, and therefore, $u$ is a bounded  solution to $MA(\Omega, \phi,0)$. 

It remains to prove that $u\in \mathcal C^{\min(\frac{\alpha}{m}, \alpha)}(\overline\Omega)$. 
Let   $0<\delta \leq  1$. Put  
$$\Omega_\delta:=\{z\in \Omega: dist(z,\partial \Omega)>\delta\},$$
and  
$$
u_\delta(z):= \sup_{\overline {B(z, \delta)}} u , \ z\in  \overline  \Omega_\delta .
$$ 
Let $z\in \partial \Omega_\delta$ and $w\in \overline{ B(z,\delta)}$. Choose $\xi \in \partial \Omega$ such that $|z-\xi|<2\delta$. 
From \eqref{2.9.30.6} we have 
\begin{align*}
u(w) - u(z)
\leq \phi(\xi) - h_\xi (w)   - u(z) 
\leq  -h_\xi(z) - h_\xi(w).
\end{align*}
For simplicity we use the notation  $\lesssim$   to denote that the inequality is up to a positive constant independent of $z,w,\xi,\delta$.
Since $\rho \in \mathcal C^{\frac{2}{m}}(\overline{\Omega})$ and $\rho(\xi)=0$,  so 
\begin{align*}
u(w) - u(z)
&  \lesssim    [\rho(\xi) -\rho(z)]^{\frac{\alpha}{2}} + |z-\xi|^{ \alpha} +   [\rho(\xi)-\rho(w )]^{\frac{\alpha}{2}} + |w-\xi|^{\alpha}
\\& \lesssim  |z-\xi| ^{\frac{ \alpha}{m}}  + |z-\xi|^{\alpha} + |w-\xi| ^{\frac{ \alpha}{m}}  + |w-\xi|^{\alpha} 
\\& \lesssim  \delta^{\frac{ \alpha}{m}}  + \delta^{\alpha} 
\lesssim   \delta^{\min(\frac{\alpha}{m}, \alpha)} .
\end{align*}
Hence,
$$
u(w) - u(z)
 \leq B \delta^{\min(\frac{\alpha}{m}, \alpha)} ,
$$
where $B$ is a positive constant independent of  $w,z,\delta$. Therefore,
\begin{equation} \label{eq12.4.7.8}
u(z) \geq u_\delta(z) - B\delta^{\min(\frac{\alpha}{m}, \alpha)} , \ \forall z\in \partial \Omega_\delta.
\end{equation}
Now, put
$$\varphi_\delta:= 
\begin{cases}
\max (u_\delta - B\delta^{\min(\frac{\alpha}{m}, \alpha)} , u)  & \text{ on } \Omega_\delta
\\ u & \text{ on } \Omega \backslash \Omega_\delta
\end{cases}.$$ 
Then, from \eqref{eq12.4.7.8} we have  $\varphi_\delta \in  PSH(\Omega)$.  
Now, let  $\xi \in \partial\Omega$. Choose $R>0$ such that 
\begin{equation}\label{eq12.4.7.899999}
- h_\xi \geq 2M \text{ on } \overline  \Omega \backslash   B(0,R).
\end{equation} 
Since $\varphi_\delta \leq  M$ in $\Omega$ and $u=\varphi_\delta=\phi$ on $\partial\Omega$, by  \eqref{2.9.30.6} it implies that 
$$
\varphi_\delta - \phi(\xi) +h_\xi \leq 0 \text{ on } \partial (\Omega \cap B(0,R)).
$$
Hence, by   the maximum principle,  
$$
\varphi_\delta -    \phi(\xi) + h_\xi  \leq 0 \text{ on } \Omega \cap B(0,R) .
$$
Combining this with \eqref{eq12.4.7.899999}   we arrive  at 
$$
\varphi_\delta \leq    \phi(\xi) -h_\xi   \text{ on } \Omega, \forall \xi \in \partial\Omega.
$$
Therefore, by the definition of $u$ we infer that 
$$\varphi_\delta \leq u \text{ in } \Omega.$$
Thus,
$$
u_\delta 
\leq \varphi_\delta  +B\delta^{\min(\frac{\alpha}{m}, \alpha)}
\leq u +B\delta^{\min(\frac{\alpha}{m}, \alpha)} \text{ on }  \overline{\Omega}_\delta.
$$
By Proposition \ref{Prooooooooooo1234556}, it   implies that 
$u\in \mathcal C^{\min(\frac{\alpha}{m}, \alpha)} (\overline{\Omega})$.
The proof is complete.
\end{proof}

Next, we will prove  that there exists a bounded solution to $MA(\Omega,\phi,f)$ in pseudoconvex domains of plurisubharmonic type $m$.

\begin{lemma} \label{pro1}
Let $m>0$ and let  $\Omega$ be a pseudoconvex domain of plurisubharmonic type $m$. Let   $\rho$ be as in Definition \ref{de1} and let    $\phi \in \mathcal C^{\alpha} (\partial \Omega)$  with $0< \alpha \leq 2$.
Let   $u(\Omega,\phi,0)$ be as in Lemma \ref{lem1}.
Then, for every $p>1$ and for every   $0\leq f \in L^p(\Omega ) $ with   compact support in $\Omega$, there exist a  constant $A>0$  and   a bounded solution $u(\Omega,\phi,f)$ to $MA(\Omega, \phi,f)$ such that 
$$
u(\Omega,\phi,0)  +A\rho \leq u(\Omega,\phi,f)  \leq u(\Omega,\phi,0) \text{ on } \overline{\Omega}.
$$
\end{lemma}

\begin{proof} Put   $u_0:=u(\Omega,\phi,0)$. 
First, we claim that 
there exist $A>0$ and  $\psi \in PSH(\Omega) \cap L^\infty(\Omega)$ such that 
$
A\rho \leq \psi \leq 0  
$
and 
$$
(dd^c \psi)^n \geq f dV \text{ in }\Omega.
$$
Indeed, let $\delta$ be a positive  real number and let $D$ be a smoothly bounded strongly pseudoconvex domain such that 
$$ supp f \Subset  \{ \rho<-\delta\} \Subset D.$$
By Theorem 3 in \cite{Ko96} there is a  continuous solution $\psi_0$ to $MA(D, 0, f)$.
Choose $A>0$ such that   $supp f \Subset D\cap \{\psi_0> A (\rho + \delta) \}$. Put  
$$
\psi:=
\begin{cases}
\max(\psi_0  -A\delta ,  A \rho)   & \text{ on } D,
\\ A\rho & \text{ on } \Omega \backslash D .
\end{cases}
$$
It is easy to see that 
$\psi \in PSH(\Omega) \cap L^\infty  ( \Omega )$ and  $A \rho  \leq \psi \leq 0$ on $\overline{\Omega}$. Since  $supp f \Subset D\cap \{\psi> A \rho\}$, by Theorem 4.1 in \cite{KH}, 
$$
(dd^c \psi)^n 
\geq 1_{D \cap \{\psi>  A\rho \} } (dd^c \psi)^n 
=1_{D \cap \{\psi >  A\rho \} } (dd^c (\psi_0-A\delta))^n 
=   f dV \text{ in } \Omega.
$$
This proves the claim. 

Now, let $\{\Omega_j\}$ be an increasing sequence of  smoothly bounded strongly pseudoconvex domains such that    $supp f \Subset \Omega_j \Subset \Omega_{j+1} \Subset \Omega$, $\forall j\geq 1$ and  $\Omega=\bigcup_{j=1}^\infty \Omega_j$. 
By Theorem 3 in \cite{Ko96} there exist   continuous solutions $u_j $ to $MA(\Omega_j, u_0, f)$. 
Since $ u_0+ \psi \leq   u_j \leq u_0$ on $\partial \Omega_j$ and 
$$
(dd^c (u_0+ \psi ))^n \geq (dd^c  u_j)^n  \geq (dd^c  u_0)^n  ,
$$
by the comparison  principle  we have 
$$ 
u_0+  \psi \leq   u_j \leq u_0 \text{  on } \overline  \Omega_j .
$$
It follows that 
$$u_{j+1} \leq u_0 =u_j \text{  on } \partial  \Omega_{j} .$$
Again by the comparison principle,  
$$u_{j+1} \leq  u_j \text{  on } \overline  \Omega_{j} .$$
Put $u:=\lim_{j\to\infty} u_j$. Since 
$$
u_0 +A\rho \leq u_0 +\psi \leq u \leq u_0
\text{ on } \overline{\Omega}
$$
so 
$u\in PSH(\Omega)\cap L^\infty(\Omega)$ and 
$(dd^c u)^n =fdV$ in  $\Omega$. Thus, $u$ is a bounded solution to $MA(\Omega, \phi, f)$.
The proof is complete.
\end{proof}

From Theorem 3 in \cite{Ko96} and Lemma \ref{pro1} we give 

\begin{proposition}\label{prooooos}
Let $m>0$ and let $\Omega$ be a  pseudoconvex    domain of plurisubharmonic type $m$. Let   $\phi \in \mathcal C^{\alpha} (\partial \Omega)$  with $0< \alpha \leq 2$ and let  $0\leq f \in L^p(\Omega ) $ with   $p>1$.
Assume that either  $\Omega$ is bounded or the support of $f$ is compact on $\Omega$. 
Then, there exists a bounded solution to $MA(\Omega, \phi,f)$.
\end{proposition}

Note that  the uniqueness of solutions in bounded domains implies  from Theorem 3.9 in \cite{Ce08}. On  unbounded domains,   the uniqueness of solutions is still open.

\section{H\"older continuity}

First, we prove the following lemma.

\begin{lemma}\label{lem3}
Let $m>0$ and let  $\Omega$ be a pseudoconvex   domain of plurisubharmonic type $m$. Let $p>1$ and   let  $0\leq f \in L^p(\Omega ) $ with compact support on $\Omega$.
Assume that  $u \in PSH(\Omega )\cap L^\infty(\Omega)$ such that $(dd^c u)^n =fdV$ in $\Omega$.   Then, for every 
$$0\leq  \gamma <\frac{1}{1+ \frac{n p}{p-1}},$$
there exists a positive constant $A_\gamma$ such that 
$$\sup_\Omega (v-u) \leq A_\gamma \left( \int_{supp f}  |u-v| dV\right) ^\gamma,$$
for every $v\in PSH(\Omega)$ with 
$\{u\leq v-\varepsilon\}\Subset \Omega$, $\forall\varepsilon>0$. 
\end{lemma} 

\begin{proof}  
The proof is almost the same as the one  given by \cite{GKZ08}. For convenience to readers, we sketch the proof of the lemma. 
Let $\rho$ be as in Definition \ref{de1}.  
Fix $v\in PSH(\Omega)$ such that 
$$\{u\leq v-\delta\}\Subset \Omega,\  \forall\delta>0.$$ 
For simplicity, by $\lesssim$    denote   inequalities up to a positive constant   independent of $v$. 
Put   
$$\varepsilon := \left( \int_{supp f}  |u-v| \beta ^n\right) ^\gamma $$ 
and 
$$\tau:=\frac{ \gamma q}{1- \gamma (nq+1)} ,$$ 
where  $q= \frac{p}{p-1}$  and $\beta:= dd^c |z|^2 $.
Without loss of generality we can assume that $0<\varepsilon<+\infty$.  
The proof is split into two steps.

{\em Step 1.} We prove that 
there exists a  constant $B_\tau>0$ such that
\begin{align*}
\int_{E} dV \leq B_\tau \left[ {\rm Cap}(E,\Omega) \right]^{q (1+n \tau)} ,
\end{align*}
for every Borel subset $E\subset supp f$,
where
$$
{\rm Cap} (E, \Omega):= \sup\{\int_E (dd^c \varphi)^n :\varphi \in PSH(\Omega), -1 \leq \varphi \leq 0 \}.
$$
 Indeed, let $\delta >0$ and let $D$ be a bounded hyperconvex domain such that 
$$supp f \Subset    \{\rho<-2\delta\} \Subset    \{\rho<-\delta\}  \Subset D \Subset \Omega.$$
Assume that   $\varphi \in PSH(D)$ with $-1\leq \varphi \leq 0$ and define 
$$
\psi : = 
\begin{cases}
\max(\delta \varphi, \rho+\delta ) & \text{ on } D,
\\ \rho +\delta & \text{ on } \Omega \backslash D.
\end{cases}
$$
Then,  $\psi \in PSH(\Omega)$ and $-\delta \leq \psi \leq \delta$ in $\Omega$. Let $E\subset supp f$ be a Borel set. Since 
$$\varphi = \frac{\psi}{\delta} \text { in } \{\rho<-2\delta\},$$
by Theorem 4.1 in \cite{KH} we arrive at 
\begin{align*}
\int_E (dd^c \varphi)^n 
& =  \int_E \left  (dd^c \frac{\psi}{\delta} \right )^n
\\ & =  2^n \int_E \left (dd^c \frac{\psi -\delta}{2\delta}\right  )^n 
 \leq 2^n  {\rm Cap}(E,\Omega).
\end{align*}
This implies that 
\begin{align*}
{\rm Cap}(E,D)  \leq 2^n  {\rm Cap}(E,\Omega).
\end{align*}
Since $q (1+n \tau)>1$,  by Proposition 1.4 in \cite{GKZ08} there exists a constant $C_\tau>0$ independent of $E$ such that 
\begin{align*}
\int_E dV\leq C_\tau \left[ {\rm Cap} (E,D) \right]^{q (1+n \tau)} .
\end{align*}
Therefore,
\begin{align*}
\int_E dV 
\leq  B_\tau \left[ {\rm Cap}(E,\Omega) \right]^{q (1+n \tau)},
\end{align*}
where  $B_\tau =2^n C_\tau$ is a positive constant independent of $E$.

{\em Step 2.} 
Consider the decreasing right continuous function $g$ defined on $\mathbb R^+$ by 
$$g(s):=\left[{\rm Cap}(U_s,\Omega)\right]^{\frac{1}{n}}, \text{ where } U_s:= \{u-v< -2\varepsilon -s\}.$$
First, we claim that  
$$t g(s+t)\lesssim  [g(s)]^{1+n \tau} \text{ for all } t,s>0.$$ 
Indeed, fix $s,t>0$.  Let $\Omega'$ be a smoothly bounded strongly pseudoconvex domain such that $\{u-v<-\varepsilon\} \Subset \Omega' \Subset \Omega$. Since 
$$
\liminf_{\Omega' \ni z\to \partial \Omega'} (u+ 2\varepsilon-v) \geq 0,
$$
by Lemma 1.3 in \cite{GKZ08} we arrive at 
\begin{align*}
t^n [g(s+t)]^n
& = t^n {\rm Cap} (U_{s+t}, \Omega)  
\\ & \leq t^n {\rm Cap} (\{u+ 2 \varepsilon-v< -s-t\}, \Omega')
\\ & \leq \int_{\{ u+ 2 \varepsilon -v<-s\}} (dd^c u)^n 
= \int_{supp f \cap  U_s} fdV.
\end{align*}
By Step 1 and using the  H\"older inequality,
\begin{align*}
t  g(s+t) 
& \leq \left( \int_{supp f \cap U_s} fdV \right) ^{\frac{1}{n}}
\\& \lesssim  \left( \int_{supp f \cap U_s} dV \right) ^{\frac{1}{nq}} 
\\ & \lesssim [{\rm Cap} (supp f \cap  U_s,\Omega)]^{\frac{1+ n \tau}{n}}
\\ & \leq [{\rm Cap} (U_s,\Omega)]^{\frac{1+ n \tau}{n}}
=  [g(s)]^{1+n\tau}.
\end{align*}
This proves the claim, and therefore, by  Lemma 1.5 in \cite{GKZ08} we get 
$g(s) =0$ for all $s\geq s_\infty$, where
$$s_\infty \lesssim  [g(0)]^{n\tau}.$$
It follows that 
$$u-v\geq -2\varepsilon  -s_\infty \text{ on }\Omega.$$ 
Therefore,  by Lemma 1.3 in \cite{GKZ08} and using the  H\"older inequality, we get 
\begin{align*}
\sup_\Omega (v-u) 
&\leq 2\varepsilon + s_\infty \lesssim  \varepsilon +  [g(0)]^{n \tau} 
\\&  \leq  \varepsilon+  [{\rm Cap} ( \{u-v<-2\varepsilon\},\Omega')]^{\tau} 
\\& \lesssim \varepsilon +  \left[ \varepsilon ^{-n} \int_{\Omega' \cap \{u+\varepsilon -v<-\frac{\varepsilon}{2}\}}  (dd^c u)^n \right ]^{\tau}
\\& = \varepsilon +   \left(   \varepsilon ^{-n} \int_{\Omega'\cap \{u+\varepsilon -v<-\frac{\varepsilon}{2}\}}  f dV \right )^{\tau}
\\& \lesssim \varepsilon +   \left(  \varepsilon  ^{-n-\frac{1}{q}} \int_{supp f}  f | u - v|^{\frac{1}{q}} dV \right )^{\tau}
\\& \lesssim \varepsilon +   \left[ \varepsilon  ^{-n-\frac{1}{q}}  \left( \int_{supp f} | u -v | dV \right)^{\frac{1}{q}} \right ]^{\tau} 
\\& = \varepsilon +   \varepsilon  ^{(-n -\frac{1}{q} +\frac{1}{q \gamma}) \tau}  
=2  \varepsilon  .
\end{align*}
The proof is complete.
\end{proof}

\begin{lemma}\label{le14.12.2016}
Let $m>0$ and let  $\Omega$ be a pseudoconvex domain of plurisubharmonic type $m$. Let   $\rho$ be as in Definition \ref{de1} and let    $\phi \in \mathcal C^{\alpha} (\partial \Omega)$  with $0< \alpha \leq 2$. 
Let   $u(\Omega,\phi,0)$ be as in Lemma \ref{lem1}.
Then, for every $p>1$ and for every   $0\leq f \in L^p(\Omega ) $ with   compact support on $\Omega$, there exist a  constant $A>0$  and   a bounded solution $u(\Omega,\phi,f)$ to $MA(\Omega, \phi,f)$ such that 
$$
u(\Omega,\phi,0)  +A\rho \leq u(\Omega,\phi,f)  \leq u(\Omega,\phi,0) \text{ on } \overline{\Omega}.
$$  
Moreover,  $u (\Omega,\phi,f) \in \mathcal C^{\gamma}(\overline{\Omega})$ for all 
$$0<\gamma <  \min \left(\frac{ \alpha}{2m}, \frac{\alpha}{2} , \frac{1}{ 1 + \frac{np}{p-1} } \right)  .
$$  
\end{lemma}

\begin{proof}
The existence  imply from  Lemma \ref{pro1}. It remains to prove that $u (\Omega,\phi,f) \in \mathcal C^{\gamma}(\overline{\Omega})$ for all 
$$0<\gamma <  \min \left(\frac{ \alpha}{2m}, \frac{\alpha}{2} , \frac{1}{ 1 + \frac{np}{p-1} } \right)  .
$$  
Put $v:=A \rho$, $w:=u(\Omega, \phi,0) $ and $u:=u(\Omega, \phi,f)$. It is easy to see that   $v  \in \mathcal C^{\frac{2}{m}} (\overline{\Omega}) $.  
By Lemma \ref{lem1} and Lemma \ref{pro1} we have $ w \in \mathcal C^{\min(\frac{\alpha}{m},\alpha)} (\overline{\Omega})$ and 
\begin{equation} \label{24.7.13.12.2016}
v+w \leq u \leq w \text{ on } \overline \Omega . 
\end{equation}  
Fix $0 < \gamma <\min \left(\frac{ \alpha}{2m}, \frac{\alpha}{2} , \frac{1}{ 1 + \frac{np}{p-1} } \right)$. Choose $\delta_0 \in (0,1)$ such that  $supp f + \mathbb B(0,\sqrt{ \delta_0} ) \Subset \Omega$. 
Let   $\delta\in (0, \delta_0)$ and  let $\varphi\in PSH(\Omega)$. We define 
$$\Omega_\delta:=\{z\in \Omega: dist(z,\partial \Omega)>\delta\} ,$$  
$$\varphi_\delta(z):= \sup_{ {\mathbb B(z, \delta)}} \varphi, \ z\in  \overline  \Omega_\delta$$ 
and  
$$\hat{\varphi}_\delta(z) :=\frac{1}{\sigma_{2n} \delta^{2n}} \int_{  {\mathbb B(z, \delta)}} \varphi  dV,\ z\in  \overline  \Omega_\delta,$$
where $\sigma_{2n}$ is the volume of the unit ball in $\mathbb C^n$.   For simplicity we use the notation  $\lesssim$   to denote that the inequality is up to a positive constant independent of $z,\xi,\delta$. 
Since $ 2\gamma  <   \min(\frac{\alpha}{m},\alpha) \leq \frac{2}{m}$ so  
$v,w  \in   \mathcal C^{2\gamma} (\overline{\Omega})$.
Hence, 
$$
w(\xi) -w(z) 
\lesssim   |z-\xi|^{2\gamma} 
\lesssim \delta^{2\gamma},
$$
for every $ z\in \overline{\Omega}_\delta$ and for every $\xi \in \overline{ \mathbb B(z,\delta)}$.  Therefore,
$$
w_\delta -w
\lesssim  \delta^{2\gamma} 
\text{ on } \overline  \Omega _\delta.
$$
By the  the hypotheses, we get 
\begin{equation} \label{eq5.28.8.12.201612}
u_\delta   - u
\leq w_\delta  - w-v
\lesssim -v + \delta^{2\gamma} 
\text{ on } \overline  \Omega _\delta.
\end{equation}
Since  $v=0$ on $\partial\Omega$ and $v\in \mathcal C^{2\gamma} (\overline{\Omega})$, so 
$$
|v| \lesssim \delta^{2\gamma}
\text{ on } \partial \Omega_\delta.
$$
Combining this with \eqref{eq5.28.8.12.201612} we arrive at 
$$
u_\delta   - u  
\lesssim  \delta^{2\gamma} 
\text{ on } \partial \Omega_\delta. 
$$
It  implies that there is  a positive constant $A$  independent of $\delta$  such that  
\begin{equation} \label{eq11.12.201612} 
w_\delta \leq w +A \delta^{2\gamma} , \
v_\delta \leq v+A\delta^{2\gamma} \text{ on } \overline  \Omega_\delta  \text{ and } 
u_\delta   \leq   u
+  A \delta^{2\gamma}
\text{ on } \partial \Omega_\delta .
\end{equation}
Since $u\in PSH(\Omega)$ and $supp f +\mathbb B(0,\sqrt{\delta_0}) \Subset \Omega$, so 
\begin{equation}\label{7.10.12.12.2016}
\int_{supp f + \mathbb B(0,\sqrt{\delta_0})}  \Delta u  
<+\infty.
\end{equation}
From Jensen's formula and using polar coordinates, for every $z\in \Omega_{\sqrt \delta}$,
\begin{align*}
\hat{u}_{\sqrt \delta}(z) -u (z) =\frac{1}{\sigma_{2n-1} \delta^{n}} \int_0^{\sqrt \delta} r^{2n-1} dr \int_0^r t^{1-2n}   \left( \int_{|\xi - z|\leq t}   \Delta u(\xi)     \right) dt.
\end{align*} 
Hence, by \eqref{7.10.12.12.2016} and using Fubini's theorem we infer at  
\begin{equation} \label{13.12.2016c12}
\begin{split}
\int_{supp f } [\hat{u}_{\sqrt \delta}(z) -u(z)]&  dV(z)
 =  \frac{1}{\sigma_{2n-1} \delta^{n}}  \int_{supp f} \left[ \int_0^{\sqrt\delta} r^{2n-1} dr \int_0^r t^{1-2n}  \right.
\\&  \ \ \ \  \ \ \ \times \left. \left( \int_{|\xi -z|\leq t}  \Delta u(\xi)     \right)dt \right] d V(z) 
\\  & \lesssim    \delta^{- n}   \int_0^{\sqrt\delta} r^{2n-1} dr \int_0^r t^{1-2n}   
\\&  \ \ \ \  \ \ \ \times  \left[ \int_{|\xi|\leq t}  \left ( \int_{supp f} \Delta u(z + \xi )     \right)dV(\xi) \right] d t 
\\ & \lesssim  \delta^{-n}  \int_0^{\sqrt\delta} r^{2n-1} dr \int_0^r t^{1-2n}  
\\&  \ \ \ \  \ \ \ \times \left[ \int_{|\xi|\leq t} \left(  \int_{supp f + \mathbb B(0,\sqrt{\delta_0})}  \Delta u   \right) dV(\xi) \right] dt 
\\ & \lesssim  \delta  .
\end{split} 
\end{equation}
Now, we set 
$$U_{\sqrt  \delta}:= 
\begin{cases}
\max (u_{\sqrt \delta} - 4 A \delta^{\gamma}, u)  & \text{ on } \Omega _{\sqrt \delta} 
\\ u & \text{ on } \Omega \backslash  \Omega _{\sqrt \delta} 
\end{cases}$$
and 
$$\hat{U}_{\sqrt \delta} := 
\begin{cases}
\max (\hat{u}_{\sqrt \delta} - 4 A \delta^{ \gamma}, u)  & \text{ on } \Omega _{\sqrt \delta} 
\\ u & \text{ on } \Omega \backslash  \Omega _{\sqrt \delta}.
\end{cases}$$
Then,  $U_{\sqrt \delta}, \hat{U}_{\sqrt \delta} \in  PSH(\Omega)$.
Since $v+w \leq u \leq w $ on $\overline{\Omega}$ so   by  \eqref{eq11.12.201612} we get
\begin{align*} 
u_{\sqrt \delta}
& \leq w_{\sqrt \delta}
\leq w +A \delta^{\gamma} 
\\& \leq u -v +A \delta^{\gamma} 
\leq u+ 4A \delta^{ \gamma} 
\end{align*}
on $\overline{\Omega}_{\sqrt \delta} \cap \{v\geq - 3 A \delta^{ \gamma} \}$.  From the hypotheses,  
$$
\{u< \hat U_{\sqrt \delta}\}  
 \subset  \Omega_{\sqrt \delta} \cap \{ u< u_{\sqrt \delta}- 4 A \delta^{ \gamma}\}
 \subset \Omega_{\sqrt \delta} \cap \{v<- 3 A\delta^{ \gamma}\} \Subset \Omega.
$$
Moreover, since $\hat{U}_{\sqrt \delta} \leq \hat{u}_{\sqrt \delta}$ in $\Omega_{\sqrt \delta}$ and $0<\gamma < \frac{1}{1 + \frac{np}{p-1}} $, by Lemma \ref{lem3} and using  \eqref{13.12.2016c12} we have 
\begin{equation} \label{1.5.2017}
\begin{split}
\sup_{\Omega }  ( \hat{U}_{\sqrt \delta} -u)  
& \lesssim \left( \int_{supp f}   |\hat{U}_{\sqrt \delta} -u| dV\right) ^\gamma
\\& \leq \left( \int_{supp f}   |\hat{u}_{\sqrt \delta} -u| dV\right) ^\gamma 
\lesssim \delta^\gamma.
\end{split} 
\end{equation} 
Lemma 4.3 in \cite{Hie10} implies that 
$$
|\hat u_\delta (x)- \hat  u _\delta (y)| \leq \frac{ \|u\|_{L^\infty(\Omega)}  |x -y|}{ \delta}, \ \forall x,y \in \Omega_\delta.
$$ 
Let $z\in \Omega_{2\sqrt{\delta}} \subset \Omega_{2\delta}$. Since $u \leq \hat u_\delta$ in $\Omega_\delta$, we get 
\begin{align*}
u_\delta(z)&  = \sup_{t\in \mathbb B(0,\delta)} u(z+t)
\leq \sup_{t\in \mathbb B(0,\delta)} \hat u_{\sqrt{\delta}} (z+t)
\\ & \leq  \hat u_{\sqrt{\delta}}(z) + \sqrt{\delta} \|u\|_{L^\infty(\Omega)} .
\end{align*}
Hence, by  \eqref{1.5.2017} we arrive at
\begin{equation} \label{13.12.2016c123456}
\begin{split}
\sup_{\Omega_{2 \sqrt \delta} }  (u_\delta -u) 
& \lesssim \sup_{\Omega_{ 2\sqrt \delta}}   (\hat{u}_{\sqrt \delta}  -u)  +   \sqrt \delta  
\\ & \lesssim \sup_{\Omega }  ( \hat{U}_{\sqrt \delta} -u)  +  \delta^{ \gamma} +  \sqrt{\delta}  
\\ & \lesssim \delta^\gamma.
\end{split} 
\end{equation} 
Now, since $v=0$ on $\partial\Omega$ and $v \in  \mathcal C^{2 \gamma}(\overline{\Omega})$ so   by \eqref{24.7.13.12.2016} we get 
$$
w - \delta^\gamma \lesssim  u \lesssim w  \text{ on } \Omega \backslash \Omega _{2\sqrt \delta}.
$$
Moreover, since $w\in  \mathcal C^{2 \gamma}(\overline{\Omega})$, this follows that 
$$
u_\delta \lesssim u +\delta^\gamma   \text{ on } \Omega_{\delta} \backslash \Omega _{2\sqrt \delta}.
$$
Combining this with \eqref{13.12.2016c123456} and using  Proposition \ref{Prooooooooooo1234556}, we infer at 
$u \in  \mathcal C^{ \gamma}(\overline{\Omega})$.
The proof is complete. 
\end{proof}

Now, we give the proof of Theorem \ref{the1}.
 
\begin{proof}   
When   the support of $f$ is compact on $\Omega$,   the statement follows  from  Lemma \ref{le14.12.2016}. We now assume that $\Omega$ is bounded. By Theorem 3 in \cite{Ko96} and Theorem 3.9 in \cite{Ce08}, there exists a unique solution $u$ to $MA(\Omega,\phi,f)$. It remains to prove that $u \in \mathcal C^{\gamma}(\overline{\Omega})$  for all  
$$
0<\gamma < \gamma_{m,\alpha,p}:= \min \left(\frac{ \alpha}{2m}, \frac{\alpha}{2} , \frac{1}{2m(1 + \frac{np}{p-1})}  , \frac{1}{2(1 + \frac{np}{p-1})} \right)  .
$$  
Fix $ \gamma \in (0, \gamma_{m,\alpha, p})$.
Let $D$ be a bounded strictly pseudoconvex domain such that $\Omega \Subset D$. 
Since $D$ is pseudoconvex domain of plurisubharmonic type $2$, by Lemma \ref{le14.12.2016} there exists a $\gamma'$-H\"older solution  $u'$   to $MA(D, 0,1_\Omega f)$    for all  
$$
0<\gamma' < \frac{1}{1 + \frac{np}{p-1}} .
$$ 
Applying Lemma \ref{lem1}, there exists a $\min(\frac{\gamma'}{m}, \gamma')$-H\"older solution  $\phi'$   to $MA(\Omega, -u' ,0)$    for all  
$$
0<\gamma' < \frac{1}{1 + \frac{np}{p-1}} .
$$ 
Put 
$v:=u'+\phi'$. Then, $v\in PSH(\Omega)$, $v=0$ on $\partial\Omega$ and $v\in  \mathcal C^{\min(\frac{\gamma'}{m}, \gamma')}(\overline{\Omega})$ for all 
$$ 0<\gamma' < \frac{1}{1 + \frac{np}{p-1}}.$$
Again by Lemma \ref{lem1}, there exists a $\min(\frac{\alpha}{m}, \alpha)$-H\"older solution  $w$   to $MA(\Omega, \phi ,0)$. Since $v+w=u=w$ on $\partial\Omega$ and 
$$(dd^c (v+w))^n 
\geq (dd^c u')^n 
= (dd^c u)^n 
\geq (dd^c w)^n \text{ in } \Omega,$$ 
by the comparison principle we infer at 
$$
v+w \leq u \leq w \text{ on } \overline{\Omega}.
$$
Let   $\delta\in (0,1)$ and let  $\Omega_\delta$, $u_\delta$,  $\hat u_\delta$ be as in proof of Lemma \ref{le14.12.2016}.  
For simplicity we use the notation  $\lesssim$   to denote that the inequality is up to a positive constant independent of $z,\xi,\delta$. 
Since 
$$0<2\gamma <2\gamma_{m,\alpha,p}= \min \left( 
\min \left(\frac{ \alpha}{m},  \alpha  \right) , 
\min \left(  \frac{1}{m(1 + \frac{np}{p-1})}  , \frac{1}{1 + \frac{np}{p-1}} \right)  \right) $$ 
so $v,w \in \mathcal C^{2\gamma} (\overline{\Omega})$. 
Hence, 
$$
w(\xi) -w(z) 
\lesssim   |z-\xi|^{2\gamma} 
\lesssim \delta^{2\gamma},
$$
for every $ z\in \overline{\Omega}_\delta$ and for every $\xi \in \overline{ \mathbb B(z,\delta)}$.  Therefore,
$$
w_\delta -w
\lesssim  \delta^{2\gamma} 
\text{ on } \overline  \Omega _\delta.
$$
By the  the hypotheses, we get 
\begin{equation} \label{eq5.28.8.12.2016}
u_\delta   - u
\leq w_\delta  - w-v
\lesssim -v + \delta^{2\gamma} 
\text{ on } \overline  \Omega _\delta.
\end{equation}
Since  $v=0$ on $\partial\Omega$ and $v\in \mathcal C^{2\gamma} (\overline{\Omega})$, so 
$$
|v| \lesssim \delta^{2\gamma}
\text{ on } \partial \Omega_\delta.
$$
Combining this with \eqref{eq5.28.8.12.2016} we arrive at 
$$
u_\delta   - u  
\lesssim  \delta^{2\gamma} 
\text{ on } \partial \Omega_\delta. 
$$
It  implies that there is  a   constant $A\geq \| v\|_{\mathcal C^{2\gamma} (\overline{\Omega})} +\| w\|_{\mathcal C^{2\gamma} (\overline{\Omega})} $  independent of $\delta$  such that  
\begin{equation} \label{eq11.12.2016} 
w_\delta \leq w +A \delta^{2\gamma} , \
v_\delta \leq v+A\delta^{2\gamma} \text{ on } \overline  \Omega_\delta  \text{ and } 
u_\delta   \leq   u
+  A \delta^{2\gamma}
\text{ on } \partial \Omega_\delta .
\end{equation}
It follows that 
\begin{equation} \label{eq11.12.2016b} 
\{v<- 3A\delta^{2\gamma}\} + \mathbb B(0,\delta) \subset    \{v<- 2 A\delta^{2\gamma}\}.
\end{equation}
Now, we set 
$$U_{ \sqrt \delta} := 
\begin{cases}
\max (u_{ \sqrt \delta} - 4 A \delta^{ \gamma}, u)  & \text{ on } \Omega _{ \sqrt \delta}
\\ u & \text{ on } \Omega \backslash  \Omega _{ \sqrt \delta}
\end{cases}$$
and 
$$\hat{U}_{ \sqrt \delta}:= 
\begin{cases}
\max (\hat{u}_{ \sqrt \delta} - 4 A \delta^{ \gamma}, u)  & \text{ on } \Omega _{ \sqrt \delta}
\\ u & \text{ on } \Omega \backslash  \Omega _{ \sqrt \delta}.
\end{cases}$$
From \eqref{eq11.12.2016} we have  $U_{ \sqrt \delta}, \hat{U}_{ \sqrt \delta}\in  PSH(\Omega)$.
Since $v+w \leq u \leq w $ on $\overline{\Omega}$ so again  by  \eqref{eq11.12.2016} we get
\begin{align*} 
u_{ \sqrt \delta}
& \leq w_{ \sqrt \delta}
\leq w +A \delta^{ \gamma} 
\\& \leq u -v +A \delta^{ \gamma} 
\leq u+ 4A \delta^{ \gamma} 
\end{align*}
on $\overline{\Omega}_\delta \cap \{v\geq - 3 A \delta^{ \gamma} \}$. Hence,  
\begin{equation} \label{eq11.12.2016c} 
\{u< \hat U_{ \sqrt \delta}\}  
 \subset  \Omega_{ \sqrt \delta}\cap \{ u< u_{ \sqrt \delta} - 4 A \delta^{ \gamma}\}
 \subset \Omega_{ \sqrt \delta} \cap \{v<- 3 A\delta^{ \gamma}\} \Subset \Omega.
\end{equation}
From Jensen's formula and using polar coordinates, for every $z\in \Omega_{ \sqrt \delta}$,
\begin{align*}
\hat{u}_{ \sqrt \delta} (z) -u (z) =\frac{1}{\sigma_{2n-1} \delta^{n}} \int_0^{ \sqrt \delta} r^{2n-1} dr \int_0^r t^{1-2n}   \left( \int_{|\xi - z|\leq t}   \Delta u(\xi)  dV(\xi)  \right) dt.
\end{align*} 
By Fubini's theorem,  using \eqref{eq11.12.2016b} and  \eqref{eq11.12.2016c}  we infer at  
\begin{equation*} \label{13.12.2016c}
\begin{split}
\int_{\{u< \hat U_{ \sqrt \delta} \} } ( \hat{u}_{ \sqrt \delta} -u) dV 
& \lesssim   \delta^{- n}  \int_0^{{ \sqrt \delta}} r^{2n-1} dr \int_0^r t^{1-2n}  
\\&  \ \ \ \  \ \ \ \times \left[ \int_{|\xi|\leq t} \left(  \int_{\{u< \hat U_{ \sqrt \delta}\} + \mathbb B(0,{ \sqrt \delta})}  \Delta u dV \right) dV(\xi) \right] dt 
\\ & \lesssim  \delta   \int_{\{v < - 2 A \delta^{ \gamma}\}}  \Delta u dV
\\& \lesssim  \delta  \int_{\{v < - 2 A \delta^{ \gamma}\}}  dd^c  u \wedge (dd^c |z|^2)^{n-1}.
\end{split} 
\end{equation*}
Since $v+w\leq u \leq v$ on $\overline{\Omega}$, $v=0$ on $\partial \Omega$, $v\in \mathcal C^{2\gamma}(\overline{\Omega})$ and $0\leq  2\gamma, \delta \leq 1$ so 
\begin{align*}
\{v < -2 A \delta^{ \gamma}\} 
\subset \{2v+w +2A\delta^{\gamma} <u\} \subset \{v<- A\delta^{\gamma}\} \subset \Omega_{ \sqrt \delta}
\Subset \Omega .
\end{align*}
Therefore, by the comparison principle,
\begin{equation*} \label{13.12.2016}
\begin{split}
\int_{\{u< \hat U_{ \sqrt \delta}\} } ( \hat{u}_{ \sqrt \delta} -u) dV  
& \lesssim   \delta  \int_{\{2v+w +2 A\delta^{\gamma} <u\} }  dd^c  u \wedge (dd^c |z|^2)^{n-1}
\\ & \leq   \delta   \int_{\{2v+w + 2 A\delta^{2\gamma} <u\} }  dd^c(2v+w + 2 A\delta ^{ \gamma} )   \wedge (dd^c |z|^2)^{n-1}
\\ & \lesssim   \delta  \int_{\Omega_{ \sqrt \delta} }  dd^c  h  \wedge (dd^c |z|^2)^{n-1},
\end{split}
\end{equation*}
where $h=v+w \in PSH(\Omega) \cap \mathcal C^{2\gamma}(\overline{\Omega})$. Since $\Omega$ is bounded, so 
$$ \int_\Omega dV <+\infty.$$ 
Again,  applying Jensen's formula and  Fubini's theorem, we obtain 
\begin{align*}
\int_{\{u< \hat U_{ \sqrt \delta}\} } ( \hat{u}_{ \sqrt \delta} -u) dV 
& \lesssim   \delta  \int_{\Omega_{ \sqrt \delta} }  dd^c  h \wedge (dd^c |z|^2)^{n-1}
\\ &  \lesssim   \delta^{- n}  \int_0^{\frac{{ \sqrt \delta}}{2}} r^{2n-1} dr \int_0^r t^{1-2n}  
\\&  \ \ \ \  \ \ \ \times \left[ \int_{|\xi|\leq t} \left(  \int_{\Omega_{ \sqrt \delta} }  \Delta h(z) dV(z) \right) dV(\xi) \right] dt 
\\  & \lesssim    \delta^{- n}   \int_0^{\frac{{ \sqrt \delta}}{2}} r^{2n-1} dr \int_0^r t^{1-2n}   
\\&  \ \ \ \  \ \ \ \times  \left[ \int_{|\xi|\leq t}  \left ( \int_{\Omega_{\frac{{ \sqrt \delta}}{2}}} \Delta h(z+\xi) dV (z)   \right)dV(\xi) \right] d t 
\\ & = \delta^{- n}  \int_{\Omega_{\frac{ \sqrt \delta}{2}}} \left[ \int_0^{\frac{ \sqrt \delta}{2} } r^{2n-1} dr \int_0^r t^{1-2n}  \right.
\\&  \ \ \ \  \ \ \ \times \left. \left( \int_{|\xi -z|\leq t}  \Delta u(\xi) dV (\xi)   \right)dt \right] d V(z) 
\\ &  \lesssim  \int_{\Omega_{\frac{ \sqrt \delta}{2}}} ( \hat{h}_{\frac{ \sqrt \delta}{2}} -h) dV 
 \lesssim \delta^{ \gamma} \int_\Omega dV
\lesssim \delta^{ \gamma}. 
\end{align*} 
Combining this with \eqref{eq11.12.2016c} , Theorem 1.1 in \cite{GKZ08}   and  Lemma 4.3 in \cite{Hie10}  we arrive at
\begin{equation*} 
\begin{split}
\sup_{\Omega_{ 2\sqrt \delta} }  (u_\delta -u) 
& \lesssim \sup_{\Omega_{2 \sqrt \delta}}   (\hat{u}_{\sqrt \delta}  -u)  +   \sqrt \delta  
\\ & \lesssim \sup_{\Omega }  ( \hat{U}_{2\sqrt \delta} -u)  +  \delta^{ \gamma} +  \sqrt{\delta}  
\\ & \lesssim \left( \int_{\{u < \hat U _{2\sqrt \delta} \} }   |\hat{U}_{2\sqrt \delta} -u| dV\right) ^\gamma +\delta^\gamma 
\\ & \lesssim \delta^\gamma.
\end{split} 
\end{equation*} 
By Proposition \ref{Prooooooooooo1234556}, we obtain  $u \in \mathcal C^\gamma (\overline{\Omega})$. 
The proof is complete.
\end{proof}

\end{document}